\documentclass[a4paper,11pt]{amsart}
\usepackage[T1]{fontenc}
\usepackage[english]{babel}
\usepackage[cp1252]{inputenc}
\usepackage{amsthm}
\usepackage{amsmath}
\usepackage{amsfonts}
\usepackage{stmaryrd}
\usepackage{amssymb}
\usepackage{hyperref}
\usepackage{indentfirst}
\usepackage{amssymb}
\usepackage{eufrak}
\usepackage{mathrsfs}
\usepackage{xypic}
\usepackage[pdftex]{graphicx}

\newcommand{\Log}{\mathrm{Log}\,}

\newtheorem{theorem}{Theorem}[section]

\newtheorem{proposition}{Proposition}[section]

\newtheorem{lemma}{Lemma}[section]

\begin{document}
\title[Amoeba  finite basis of solutions to a system of $n$ polynomials]{Amoeba finite basis of solutions to a system of $n$ polynomials in $n$ variables}
\author{Mounir Nisse}
\date{}
 
\address{School of Mathematics KIAS, 87 Hoegiro Dongdaemun-gu, Seoul 130-722, South Korea.}
\email{\href{mailto:mounir.nisse@gmail.com}{mounir.nisse@gmail.com}}
\thanks{This research  is partially supported by NSF MPS grant DMS-1001615, and Korea Institute for Advanced Study (KIAS), Seoul, South Korea.}

\begin{abstract}
We show that  the amoeba of a complex algebraic variety defined as
the solutions to a generic system of $n$ polynomials in $n$ variables has a finite basis. In other words, it is the intersection of finitely many hypersurface amoebas.  Moreover, we give an upper bound of the size of the basis in terms of $n$ and the mixed volume $\mu$ of the Newton polytopes of the polynomial equations of the system. Also, we give an upper bound of the degree of the basis elements in terms of $\mu$.
\end{abstract}

\maketitle


\section{Introduction}

 Let  $K$ be
an algebraically closed field  equipped with a non-trivial real valuation $\nu : K \rightarrow \mathbb{R} \cup \{ \infty \}$, the tropical variety (or non-Archimedian amoebas) $\mathcal{T}rop (\mathcal{I})$ of an ideal $\mathcal{I}\subset K[x_1,\ldots, x_n]$ is defined as the topological closure of the set
$$
\nu (V(\mathcal{I})):= \{(\nu (x_1), \ldots , \nu (x_n))\,|\, (x_1,\ldots ,x_n)\in  V(\mathcal{I})\} \subset \mathbb{R}^n,
$$
where $V(\mathcal{I})$ denotes the zero set of $\mathcal{I}$ in $(K^*)^n$ (see for example  \cite{MS-09}).
A tropical basis for $\mathcal{I}$ is a generating set $\mathscr{B} = \{ g_1,\ldots , g_l\}$ of $\mathcal{I}$ such that
$$
 \mathcal{T}rop (\mathcal{I}) = \bigcap_{j=1}^l \mathcal{T}rop (\mathcal{I}_{g_j}),
$$
 where $\mathcal{I}_{g_j}$ denotes the  principal ideal generated by the  polynomial $g_j$. 
Bogart, Jensen, Speyer, Sturmfels, and Thomas \cite{BGSS-07} gave a lower bounds on the size of such bases  when $K$ is the field of Puiseux series $\mathbb{C}((t))$ and the ideal  $\mathcal{I}$ is linear with constant coefficients.  In \cite{HT-09}, Hept and Theobald showed  that any tropical variety has a tropical basis.
 The Archimedean amoeba of a subvariety of the complex torus $(\mathbb{C}^*)^n$
 is its image under the coordinatewise logarithm map. Archimedean Amoebas were introduced by Gelfand, Kapranov, and Zelevinsky in 1994 \cite{GKZ-94}. The purpose of this Note is to show that Held and Theobald's above theorem has an analogue  for Archimedian amoebas of codimension $n$. 

\vspace{0.3cm}

An algebraic variety $V\subset (\mathbb{C}^*)^n$ of codimension $n$ which is the solutions of a system of $n$ polynomial equations
 is  {\em generic} if   the Jacobian of the logarithmic map restricted to  all points of $V$   has maximal rank  i.e.,  equal to $n$.
 If the defining ideal $\mathcal{I}(V)$ of $V$ is generated by the set of polynomials  $\{ g_i\}_{i=1}^l$ with the following properties:
%
%
%
%
%


\begin{itemize}
\item[(i)]\,  $\mathscr{A}(V) = \bigcap_{i=1}^l \mathscr{A}(V_{g_i})$;
\item[(ii)]\, $\mathscr{A}(V)\subsetneq \bigcap_{i\in \{ 1,\ldots ,l\}\setminus s} \mathscr{A}(V_{g_i})$ for every $1\leq s\leq l$, \end{itemize}
then we said that $\{ g_i\}_{i=1}^l$ is an {\em amoeba basis} of $\mathscr{A}(V)$.

It was shown in \cite{N-14} that   the amoeba of a generic complex algebraic variety of codimension $1<r<n$ does not have a finite basis. In other words, it is not the intersection of finitely many hypersurface amoebas.
The aim of this note is to show that if the codimension of our variety $V$ is equal to $n$ (i.e., $V$ is a finite number of points),
then its amoeba has a finite basis. In other words,  the main  result of Hept and Theobald  in \cite{HT-09} does have an analogue for  Archimedean amoebas of generic  complex algebraic varieties  of dimension zero. 

\begin {theorem}\label{MainResult}
The amoeba of an algebraic variety $V\subset (\mathbb{C}^*)^n$ defined as the  solutions to a generic  system of $n$ polynomials $f_1,\ldots ,f_n$ in $n$ complex variables has a finite basis of length at most $(n+1)(n^{\mu -1} +1)$ where $\mu$ is the mixed volume of the Newton polytopes of the polynomials $f_1,\ldots ,f_n$. Moreover,   each element of this basis has degree at most $\mu$.
\end{theorem}

In Theorem \ref{MainResult}, if $\mu >1$ then we can decrease the length of the basis, but the price will be to increase the degree of its elements. If $\mu =1$ (i.e., $f_1,\ldots ,f_n$ is a system of a generic first order polynomials), then the length of the basis is equal to $n+1$.


\section{Proof of the main result}

 Let us recall the description of  hyperplane amoebas given by Forsberg, Passare, and  Tsikh  \cite{FPT-00}  and some of their feature
 that we will use in our proof. In Corollary 4.3 \cite{FPT-00}, Forsberg, Passare, and  Tsikh showed the following:

\begin{proposition}\label{FPT}
The amoeba $\mathscr{A}(V_f)$ of an affine complex function $f(z) = b_0+b_1z_1+\ldots +b_nz_n$ is equal to the closed (possibly empty) subset in $\mathbb{R}^n$ defined by the inequalities
\begin{eqnarray}
\log |b_0| & \leq & \log (\sum_{k=1}^n|b_k|e^{u_k})     \nonumber \\
u_j + \log |b_j| & \leq &  \log ( |b_0| + \sum_{k\ne j}^n|b_k|e^{u_k}).  \nonumber
\end{eqnarray}
\end{proposition}
\vspace{0.2cm}

\noindent Consider a generic system of $n$ polynomial equations in $n$ complex variables $z=(z_1,\ldots ,z_n)$

\vspace{0.05cm}
$$
f_1(z_1,\ldots ,z_n) = f_2(z_1,\ldots ,z_n) = \cdots = f_n(z_1,\ldots ,z_n) = 0. \quad \quad \quad (\mathscr{S})
$$
\vspace{0.05cm}

By   Bernstein's  Theorem \cite{B-75},  we know that if $\Delta_j$ is the Newton polytope of $f_j$ for $j=1, \ldots ,n$, then the number of  complex solutions of $(\mathscr{S})$ counted with multiplicity is equal to the mixed volume $\mu (\Delta_1,\ldots ,\Delta_n)$ of the $\Delta_j$'s. Let $
V := \{v^{(1)}, v^{(2)}, \ldots , v^{(l)} \}$ be the set of complex solutions of $(\mathscr{S})$, with
 $$
 \sum_{i=1}^l mult(v^{(i)}) =\mu (\Delta_1,\ldots ,\Delta_n),
 $$
 where $mult(v^{(i)})$ denotes the multiplicity of the solution $v^{(i)}$.
  Let $\mathscr{A}(V) := \{w^{(1)}, w^{(2)}, \ldots , w^{(l)} \}$ be the amoeba of $V$. If $v^{(i)}$ is equal to 
$(v^{(i)}_1,\ldots , v^{(i)}_n)\in (\mathbb{C}^*)^n$, then we denote by 
$||v^{(i)}||_0$ the  following sum:
$$
||v^{(i)}||_0 = \sum_{k=1}^n|v^{(i)}_k|.
$$

Consider the following system of $(n+1)$ polynomial equations of $n$ variables $z =(z_1,\ldots ,z_n)$, which we denote by $(\mathscr{G})$:
{\setlength\arraycolsep{2pt}
\begin{eqnarray}
g_0(z) & = & \prod_{i=1}^l \left(1- \frac{1}{||v^{(i)}||_0}\sum_{k=1}^n \frac{|v^{(i)}_k|}{v^{(i)}_k}z_k \right) =0,
                    \nonumber\\
 g_j(z) &=& \prod_{i=1}^l \left( 1 - \frac{1+||v^{(i)}||_0- |v_j^{(i)}|}{v_j^{(i)}} z_j +   \sum_{k\in\{ 1,\ldots ,n\}\setminus j} \frac{|v^{(i)}_k|}{v^{(i)}_k} z_k \right) =0,              \nonumber
\end{eqnarray}}

\noindent for $ j=1,\ldots n. $
\noindent We can check that the amoeba of each polynomial $g_j$ for $j=0,1,\ldots ,n$ is an arrangement  of $l$ hyperplane amoebas (counted with multiplicity, because it can happen that $v^{(i)}\ne v^{(m)}$ but $||v^{(i)}||_0 = ||v^{(m)}||_0$ or  $v_j^{(i)} = v_j^{(m)}$ for some $j$).
For $s=1,\ldots , l$, we denote by $\mathscr{H}_j^{(i)}$ the hyperplane with defining polynomial  $g_j^{(i)}$ which is the $i^{th}$ factor of $g_j$ for $j=0, 1,\ldots ,n$, i.e.,

$$
g_0^{(i)}(z)=1- \frac{1}{||v^{(i)}||_0}\sum_{k=1}^n \frac{|v^{(i)}_k|}{v^{(i)}_k}z_k,  \quad \textrm{for}\,\, i=1,\ldots ,l
$$
 and
$$
g_j^{(i)}(z) = 1 - \frac{1+||v^{(i)}||_0- |v_j^{(i)}|}{v_j^{(i)}} z_j +   \sum_{k\in\{ 1,\ldots ,n\}\setminus j} \frac{|v^{(i)}_k|}{v^{(i)}_k} z_k, 
$$
for  $i=1,\ldots ,l$  and  $j=1,\ldots , n$. We denote by $\mathscr{A}(V_{g_j^{(i)}})$ the hyperplane amoeba defined by $g_j^{(i)}$.

\vspace{0.2cm}

\noindent By construction, each $v^{(i)}$ is a solution of the system $(\mathscr{G})$, and then $\mathscr{A}(V) \subset \bigcap_{j=0}^n\mathscr{A}(V_{g_j})$, where $\mathscr{A}(V_{g_j})$ denotes the arrangement of hypersurface amoebas defined by the polynomial $g_j$. 
We can rewrite the system $(\mathscr{G})$ as follows
 \begin{displaymath}
 \begin{array}{cccccccc}
g_0(z) & = & g_0^{(1)}(z)&g_0^{(2)}(z)&\ldots& g_0^{(l)}(z)&=&0\\
g_1(z) & = & g_1^{(1)}(z)&g_1^{(2)}(z)&\ldots& g_1^{(l)(z)}&=&0\\
 \vdots&\vdots&\vdots&\vdots&\vdots&\vdots&\vdots&\vdots\\
 g_n(z) &=&  g_n^{(1)}(z)&g_n^{(2)}(z)&\ldots& g_n^{(l)}(z)&=&0.
\end{array} 
\end{displaymath}
Let $w\in \bigcap_{j=0}^n\mathscr{A}(V_{g_j})$  and $u\in (\mathbb{C}^*)^n$ such that $\Log (u) = w$. As we know that the amoeba defined by $g_j$ is the union of the hyperplane amoebas $\mathscr{A}(\mathscr{H}_j^{(i)})$ for $i=1,\ldots , l$, then for any $j\in\{ 0,\ldots ,n\}$ there exists $i\in \{ 1,\ldots ,l\}$ such that $w\in \mathscr{A}(\mathscr{H}_j^{(i)})$. 

\vspace{0.2cm}

\begin{lemma}\label{lemmaA}
Let $w$ be a point in the complement of the amoeba $\mathscr{A}(V)$. Then for any $i\in\{ 1,\ldots ,l\}$ there exists a nonempty subset $S_i\subset\{ 0,1,\ldots ,n\}$ such that $w\notin  \mathscr{A}(V_{g_j^{(i)}})$ for all $j\in S_i$.
\end{lemma}

\begin{proof}
If there exists $i\in\{ 1,\ldots ,l\}$ with $w\in  \mathscr{A}(V_{g_j^{(i)}})$ for all $j\in\{ 0,1,\ldots ,n\}$, then using Proposition \ref{FPT} we can check that $w=w^{(i)}$ (see \cite{SW-13} for more details).
\end{proof}

\begin{lemma}\label{lemmaB}
With the above notation, there exist a finite number of polynomial equations  $\{ h_s\}_{s=1}^r$ of degree $l$ such that:
\begin{itemize}
\item[(i)]\,  The set of solutions of the system $\{ g_j(z)= h_s(z)=0\}$ for $j=0,1,\ldots ,n$ and $s=1,\ldots ,r$ is equal to  $V$;
\item[(ii)]\, The intersection $ (\bigcap_{j=0}^n\mathscr{A}(V_{g_j}))  \cap \{ \mathscr{A}(V_{h_s})\}_{s=1}^r$ is precisely $\mathscr{A}(V) = \{ w^{(1)}, \ldots , w^{(l)}\}$.
\end{itemize}
\end{lemma}

\begin{proof}
Let  $w\in \bigcap_{j=0}^n\mathscr{A}(V_{g_j})$, then for any $j\in\{ 0,1,\ldots ,n\}$ there exists $i(j)\in \{1,\ldots ,l\}$ with $w\in \mathscr{A}(V_{g_j^{(i(j))}})$, i.e., $w\in \bigcap_{j=0}^n \mathscr{A}(V_{g_j^{(i(j))}})$. On the other hand,  suppose that $w$ is contained in the complement of the amoeba $\mathscr{A}(V)$. Hence, for any $i\in\{ 1,\ldots l\}$ there exists a subset $S_i\subset \{ 0,1,\ldots n\}$
 such that $w\notin  \mathscr{A}(V_{g_j^{(i)}})$ for all $j\in S_i$. 
 Consider the set of all arrangement of $l$ hyperplane amoebas where each of these arrangement is defined by the product of the first order polynomials  as follows: 
$$
h_{j(1),\ldots ,j(l)}(z) := g_{j(1)}^{(1)}(z)g_{j(2)}^{(2)}(z)\ldots g_{j(l)}^{(l)}(z),
$$  
 where $j(i)\in S_i$  for all $i\in\{ 1,\ldots l\}$. We can check that the number of these arrangements cannot exceed $n^{l-1}(n+1)$.  By construction, $V\subset V_{h_{j(1),\ldots ,j(l)}}$, and $w\notin  \mathscr{A}(V_{h_{j(1),\ldots ,j(l)}})$ for all index $(j(1),\ldots ,j(l))$. We claim that the intersection
 $$
 \mathscr{B} = (\bigcap_{j=0}^n\mathscr{A}(V_{g_j}))\bigcap  (\bigcap_{j(1),\ldots ,j(l)} \mathscr{A}(V_{h_{j(1),\ldots ,j(l)}}))
 $$ 
 is precisely the amoeba of $V$. Indeed, let $w \in  \mathscr{B}$ and assume $w\notin \mathscr{A}(V)$. Then for each $i\in\{ 1,\ldots ,l\}$, there exists  a nonempty subset  $S_i\subset \{ 0,1,\ldots ,n\}$ such that $w\notin \mathscr{A}(V_{g_{j(i)}^{(i)}})$ for all $j(i)\in S_i$. Otherwise,  using Proposition \ref{FPT}, the point $w$ will be equal to one of the $w^{(i)}$'s.
 Namely, there exists an index $(j(1),\ldots ,j(l))$ such that  $w$ is not in the amoeba of $V_{h_{j(1),\ldots ,j(l)}}$ (because we took all the arrangement in the intersection $ \mathscr{B}$). This is in contradiction with the fact that $w \in  \mathscr{B}$. 
The other inclusion is given by construction.
It remains to show that $V$ is equal to the set of  solutions  $V'$ of the system  $(\mathscr{S}')$ defined by 
$$
\{ g_j(z) =0\}_{j=0}^n\bigcap\{ h_{j(1),\ldots ,j(l)}(z)=0\}_{(j(1),\ldots ,j(l))}.
$$
By construction, we have the first inclusion $V\subset V'$. Without loss of generality, we can assume $l=\mu$. The number of solutions  $\nu$ (counted with multiplicity) of the system $(\mathscr{S}')$ cannot exceed $\mu$ (observe that the degree of any  equation in $(\mathscr{S}')$ is equal to $\mu$). As we know  that $V\subset V'$, then  $\nu$ is also at least equal to $\mu$. Hence, $V'$ should be equal to $V$.  
\end{proof}

\noindent {\it Proof of Theorem \ref{MainResult}}. If $l>1$, then the number of arrangements of $l$ hyperplanes amoebas cannot exceed $n^{l-1}(n+1)$ which satisfies the inequality $ n^{l-1}(n+1)\leq  n^{\mu -1}(n+1)$.  If $l=1$, then 
as we seen before, using Proposition \ref{FPT}, the set of polynomials  $\{ g_j\}_{j=0}^n$ is an amoeba basis of $\mathscr{A}(V)$. Now, Theorem \ref{MainResult}
  is an immediate consequence the last remark and  Lemma \ref{lemmaB}.

\vspace{0.3cm}
  
\noindent {\bf Example.} Let  $f_1(z)=f_2(z)=0$ be  a generic system of two polynomials in two complex variables. Let $v^{(1)} ,v^{(2)}$ be the set of solutions   of this system (for example the intersection of a line with a parabola).
Assume that   $||v^{(1)}||_0\ne ||v^{(2)}||_0$   and  $\mu =l=2$. Consider the system $(\mathscr{G})$ defined as follows:
 \begin{displaymath}
 \begin{array}{cccccc}
g_0(z) & = & g_0^{(1)}(z)&g_0^{(2)}(z)&=&0 \\
g_1(z) & = & g_1^{(1)}(z)&g_1^{(2)}(z)&=&0 \\
g_2(z) & = & g_2^{(1)}(z)&g_2^{(2)}(z)&=&0. 
\end{array} 
\end{displaymath}
Where $g_j^{(i)}$ for $j=0, 1, 2$ and $i=1,2$ are defined as before. We add the $(\mathscr{G})$ the following polynomial equations: $g_j^{(i)}g_r^{(s)}$ with $j\ne r$ and $i\ne s$. We can check that the system
$$
\mathscr{B} =  \{ g_j^{(i)}g_r^{(s)}\}_{j, r=0,1,2\, \textrm{and}\, i, s=1,2}
$$
is an amoeba basis of the set of solutions to  the system $f_1(z)=f_2(z)=0$ of length 9 and each of its element is of degree 2.



\begin{thebibliography}{amsalpha}


\bibitem[B-75]{B-75}{\sc D. Bernstein}, {\em The number of roots of a system of equations}, Functional Analysis and its
Applications {\bf 9}, (1975) 183--185.


\bibitem[BGSS-07]{BGSS-07}{\sc T. Bogart, A.N. Jensen, D. Speyer, B. Sturmfels, and R.R. Thomas}, {\em Computing tropical varieties}, J. Symb. Comp. {\bf 42}, (2007), no. 1-2, 54--73.

\bibitem[FPT-00]{FPT-00}{\sc M. Forsberg, M. Passare
and A. Tsikh}, {\em Laurent determinants and arrangements of hyperplane amoebas},
Advances in Math.  {\bf 151}, (2000), 45-70.

\bibitem[GKZ-94]{GKZ-94}{\sc I. M. Gelfand, M.
M. Kapranov and A. V. Zelevinski}, {\em Discriminants, resultants and
multidimensional determinants},
Birkh{\"a}user Boston 1994.

\bibitem[HT-09]{HT-09}{\sc K. Hept and T. Theobald}, {\em Tropical bases by regular projections}. Proc. Amer. Math. Soc., {\bf 137}(7), 2233--2241, 2009.




 
\bibitem[MS-09]{MS-09}{\sc   D. Maclagan and B. Sturmfels }, {\em  Introduction to tropical geometry}.  Book in progress, available in Bernd Sturmfels Homepage.




\bibitem[N-14]{N-14}{\sc M. Nisse}, {\em Amoeba finite basis does not exist in general}, Preprint, arXiv: 1403.3912.






\bibitem[SW-13]{SW-13}{\sc F. Schroeter and T. de Wolff}, {\em The boundary of amoebas},  preprint,  arXiv: 1310.7363.


\end{thebibliography}
\end{document}